\documentclass{daj}
\usepackage{hyperref, graphicx,amsmath,amsfonts, amssymb, amsthm, mathrsfs}

\dajAUTHORdetails{%
  title = {Unique continuation on planar graphs}, 
  author ={Ahmed Bou-Rabee ,William Cooperman, and Shirshendu Ganguly},
  plaintextauthor = {Ahmed Bou-Rabee ,William Cooperman, and Shirshendu Ganguly},
  plaintexttitle = {Unique continuation on planar graphs}, 
  copyrightauthor = {A. Bou-Rabee, W. Cooperman, S. Ganguly},
}   

\dajEDITORdetails{%
   year={2025},
   number={16},
   received={13 November 2023},   
   revised={12 September 2024},    
   published={11 September 2025},  
   doi={10.19086/da.144015},       
}   

\newtheorem{theorem}{Theorem}[section]
\newtheorem{prop}[theorem]{Proposition}

\newtheorem{lemma}[theorem]{Lemma}

\newtheorem{remark}[theorem]{Remark}

\newcommand{\G}{\mathcal{G}}
\newcommand{\V}{\mathcal{V}}
\newcommand{\E}{\mathcal{E}}
\newcommand{\cond}{\mathbf{a}}
\newcommand{\eps}{\varepsilon}
\DeclareMathOperator{\R}{\mathbb{R}}

\DeclareMathOperator{\Z}{\mathbb{Z}}
\DeclareMathOperator{\N}{\mathbb{N}}

\newcommand{\eg}{{\it e.g.}}
\DeclareMathOperator{\prev}{\mbox{prev}}
\DeclareMathOperator{\sgn}{\mathop{\bf sgn}}

\begin{document}

\begin{frontmatter}[classification=text]

\title{Unique continuation on planar graphs} 

\author[abou]{Ahmed Bou-Rabee}
\author[wcoop]{William Cooperman}
\author[sgang]{Shirshendu Ganguly}

\begin{abstract}
  We show that a discrete harmonic function which is bounded on a large portion of a periodic planar graph is constant. A key ingredient is a new unique continuation result for the weighted graph Laplacian. The proof relies on the structure of level sets of discrete harmonic functions, using arguments as in Bou-Rabee--Cooperman--Dario (2023) which exploit the fact that, on a planar graph, the sub- and super-level sets cannot cross over each other. In the special case of the square lattice this yields a new, geometric proof of the Liouville theorem of Buhovsky–Logunov–Malinnikova–Sodin (2017).
\end{abstract}
\end{frontmatter}

\section{Introduction}\label{sec:intro}
A periodic planar graph $\G = (\V, \E)$ is a graph for which there exists an embedding into the plane $\R^2$ which is invariant under translation by a two-dimensional-lattice $\mathcal{L}$.
We fix such an embedding and consider the weighted graph Laplacian on $\G$,
\[
\Delta f(u) = \sum_{w \sim u} \cond(u,w)( f(u) - f(w)) \,,
\]
where the sum is over the vertices $w \in \V$ adjacent to $u$ and the conductance, $\cond$, is a strictly positive function on the set of undirected edges $\E$. Designate a vertex closest to zero as the origin and let $B_n$ denote the graph-metric ball of radius $n$ centered at the origin. We write $\mathcal{F} := \G / 2\mathcal{L}$ to denote a (finite) quotient of the graph which, together with $\mathcal{L}$, contains all of the information needed to reconstruct $\G$. Abusing notation, we identify the vertex set $\V$ of $\G$ with its image in $\R^2$ under the embedding.

\begin{figure}
	\begin{center}
	\includegraphics[width=0.18\textwidth]{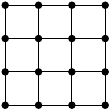} \,
	\includegraphics[width=0.1785\textwidth]{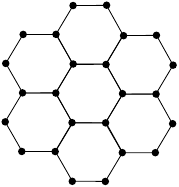} \,
	\includegraphics[width=0.172\textwidth]{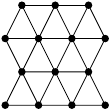} \,
	\includegraphics[width=0.19\textwidth]{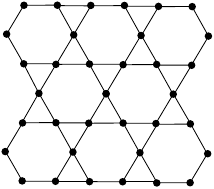} \,
	\end{center}
	\caption{Some periodic planar graphs.}
\end{figure}

Recall that any bounded harmonic function on $\G$ is constant.
Our main result is the following improvement of this Liouville theorem.
\begin{theorem}\label{theorem:liouville}
	Suppose that the conductances $\cond$ are invariant under translation by $\mathcal{L}$. Then there is some $\eps = \eps(\cond) > 0$ such that, if $\Delta f = 0$ on $\G$ and
	\begin{equation*}
		\lim_{n \to \infty} \frac{|\{|f| \leq 1\} \cap B_n|}{|B_n|} \geq 1-\eps \, ,
	\end{equation*}
	then $f$ is constant.
\end{theorem}

In the special case of $\Z^2$ with unit conductance, Buhovsky--Logunov--Malinnikova--Sodin in~\cite{buhovsky2022discrete}
established Theorem~\ref{theorem:liouville} via a delicate argument involving the polynomial structure of discrete harmonic functions on $\Z^2$. They proved two competing statements. First they showed, using the three-ball inequality, that a non-constant discrete harmonic function which is bounded on most of $\Z^2$ must grow at least exponentially. This result is quite general and, as we indicate in the appendix, also holds (albeit with a different, simpler proof) for periodic graphs. Second, using the square lattice structure, they proved an exponential upper bound for the growth of such a function, with an exponent which tends to zero as $\eps \to 0$.  One of their key observations is that a discrete harmonic function on $\Z^2$, which vanishes on two parallel diagonals, is equal (up to a sign) to a polynomial on subsequent diagonals. Unfortunately, this polynomial structure is no longer present on general planar graphs, or even on $\Z^2$ when the conductances are not constant.

Our main contribution is a completely new proof of this exponential upper bound. This may also be thought of as a unique continuation result. The argument is based on a topological property of the level sets of planar discrete harmonic functions and is therefore robust to changes in the underlying graph. Similar topological arguments were previously used by Dario and the first two authors in a study of harmonic functions on the supercritical percolation cluster~\cite{bou2023rigidity}.

Some basic features of non-crossing level sets in the plane have been used previously, \eg, in Serrin's proof of the Harnack inequality \cite{MR0081415}. We take a more quantitative approach and study the structure of the intersection points between sub- and super-level sets of zero. On the one hand, planarity forces these two sets to interleave and this ensures the existence of many distinct connected components of each set. On the other, each component intersects at least one vertex on the boundary of the domain, so there cannot be too many.
\begin{theorem}\label{theorem:uniformly-bounded}
	Suppose the conductances are uniformly bounded via $0 < \lambda < \cond < \Lambda < \infty$.
	Then there exist positive constants $\eps_0(\mathcal{F}, \Lambda/\lambda)$, $n_0(\mathcal{F})$, and $A(\mathcal{F}, \Lambda/\lambda)$
	such that for all $n \geq n_0$ and $\eps < \eps_0$,
	if $\Delta f = 0$ on $B_{2 n}$
	and $|\{x \in B_{2n} \mid  |f(x)| > 1\}| \leq \eps |B_{2n}|$, then $\max_{B_{n}} |f| \leq e^{(A \sqrt{\eps}) n} $.
\end{theorem}
Following the arguments in~\cite{buhovsky2022discrete}, the exponential lower bound, Theorem~\ref{theorem:lower-bound}, together with the exponential upper bound, Theorem~\ref{theorem:uniformly-bounded}, establishes Theorem~\ref{theorem:liouville}.

The same argument to prove Theorem~\ref{theorem:uniformly-bounded} can be used to provide a short proof of Theorem~\ref{theorem:liouville}
(with no assumptions on $\cond$ other than positivity) when the discrete harmonic function is not just bounded but also vanishes on a large portion of the graph.
\begin{theorem}\label{theorem:zero-case}
	There exist positive constants $\eps_0(\mathcal{F})$ and $n_0(\mathcal{F})$ such that if $\Delta f = 0$ on $B_{2 n}$ for $n \geq n_0$
	and $|\{x \in B_{2n} \mid |f(x)| \neq 0\}| \leq \eps |B_{2n}|$, for $\eps < \eps_0$, then $f \equiv 0$ on $B_n$.
\end{theorem}

Our theorems require the discrete two-dimensional structure as encapsulated in Lemma~\ref{lemma:topological-lemma} below.
As explained in \cite[Remark 1.3]{buhovsky2022discrete}, there are counterexamples on $\Z^d$ for $d \geq 3$ and on $\R^2$.
For instance, the function $(x,y,z) \to c^z (-1)^x  1_{\{x = y \}}$ for $c$ satisfying $c + c^{-1} = 6$ is harmonic on $\Z^3$. Planarity is also essential --- as we show in Section~\ref{sec:counterexample}, there is a family of non-planar conductances on $\Z^2$ for which Theorem~\ref{theorem:zero-case} fails.

An important application of unique continuation results of this form is Anderson localization. Ding--Smart~\cite{ding2020localization} used ideas from~\cite{buhovsky2022discrete} as input into the program of Bourgain--Kenig~\cite{bourgain2005localization} to prove localization near the edge for the Anderson--Bernoulli model on $\Z^2$.
This result was generalized to $\Z^3$ by Li--Zhang~\cite{li2022anderson}, where, as part of their argument, they proved a version of Theorem~\ref{theorem:uniformly-bounded} on the triangular lattice using ideas from~\cite{buhovsky2022discrete}.

\smallskip

Throughout $c, C$ denote positive constants that may differ in each instance. Dependence is indicated by, \eg, $C(d)$.

\section{A topological lemma}\label{sec:topology}

We give a geometric constraint on the level sets of a discrete harmonic function $f$ on a planar graph. This is the key topological observation underlying our arguments. One may think of the set $Z$ as the zero set of $f$, the set $P$ as the zero superlevel set (pluses), and the set $M$ as the zero sublevel set (minuses).
Roughly, the conclusion is that $f$ cannot have too many zeros on a circle without being identically zero in the half ball.

\begin{figure}[!ht]
	\begin{center}
		\includegraphics[width=0.35\textwidth]{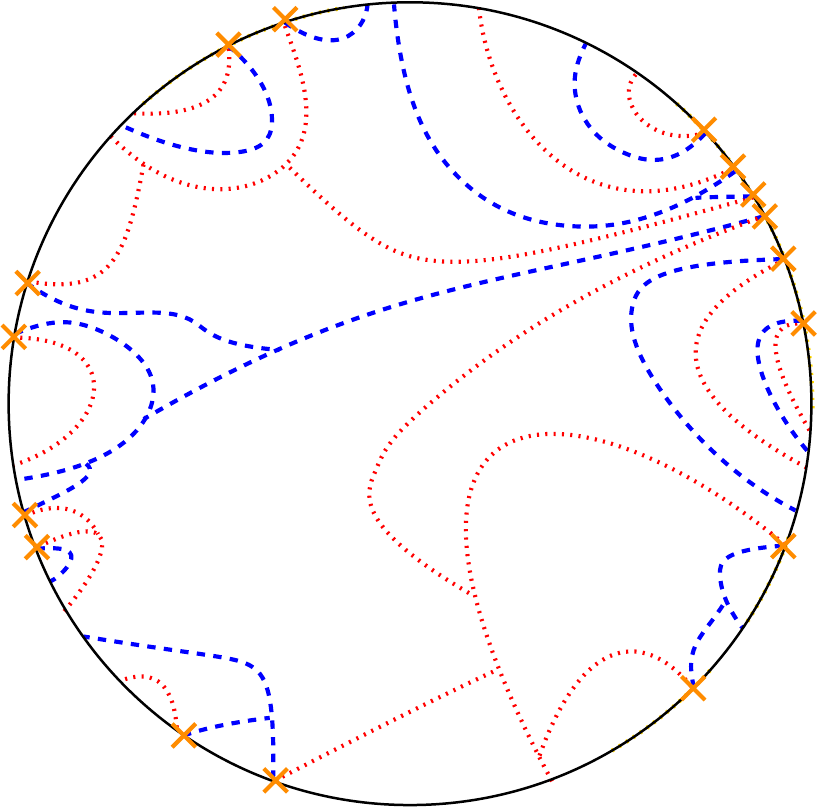}
		\caption{Sets satisfying the properties of Lemma~\ref{lemma:topological-lemma}. The cycle $\gamma$ is drawn as a black circle, the set $Z$ as dark orange crosses, and sets $P$ and $M$ as red dotted and blue dashed lines respectively.}\label{fig:inter-lines}
	\end{center}
\end{figure}

\begin{lemma}\label{lemma:topological-lemma}
	Let $\gamma := \{v_0, v_1, \dots, v_k = v_0\}$ be a cycle in $\G$ and let sets $Z, P, M \subseteq \V$ be disjoint and satisfy the following properties.
	\begin{enumerate}
		\item $Z \subseteq \gamma$.
		\item $P$ and $M$ are contained in the union of $\gamma$ and the finite component of $\G \setminus \gamma$.
		\item For every $z \in Z$, there is a path $\beta_z := \{z = w_0, w_1, \dots, w_\ell\}$ along a face of $\G$ adjacent to $z$, such that $\beta_z \cap (P \cup M) = \emptyset$ and $w_\ell$ is adjacent to vertices $p(z) \in P$ and $m(z) \in M$. If there is ambiguity, we choose $\beta_z$, $p(z)$ and $m(z)$ arbitrarily.
		\item Every connected component of the induced subgraph $\G[P]$ (respectively of $\G[M]$) intersects $\gamma \setminus Z$.
	\end{enumerate}
	Then, for a constant $\alpha_{\mathcal{F}} \in (0,1)$ depending only on the maximum size of a face in the graph $\G$, we have $|Z| \leq \alpha_{\mathcal{F}} |\gamma|$.
\end{lemma}

\begin{proof}
	For each $z \in Z$, we define $\prev_P(z) \in \gamma \setminus Z$ (resp.\ $\prev_M(z)$) to be the $v_{j}$, where $j$ is smallest in $\{1, \dots, k\}$ such that $v_{j}$ lies in the connected component of $p(z)$ in $\G[P]$ (resp. $m(z)$ in $\G[M]$).
	
	Let $Z = \{v_{j_1}, v_{j_2}, \dots, v_{j_{|Z|}}\}$, where $0 \leq j_1 < j_2 < \cdots < j_{|Z|} < k$
	and choose a maximal subset $Z' = \{v_{i_1}, v_{i_2}, \dots, v_{i_{|Z'|}}\}$ so that $p(z) \neq p(\tilde z)$ and $m(z) \neq m(\tilde z)$ and $\beta_z \cap \beta_{\tilde z} = \emptyset$ for $z \neq \tilde z \in Z'$.
	By (3), we have
	\begin{equation} \label{eq:bound-on-density}
		\frac{|Z'|}{|Z|} \geq {(\mbox{size of largest face in $\G$})}^{-4} \, .
	\end{equation}
	Indeed, for each $z \in Z$, the vertices $p(z)$,~$m(z)$, and all vertices in $\beta_z$ are within face-distance $2$ of $z$; that is, $z$ lies on the same face as a vertex which itself lies on the same face as $p(z)$ (respectively $m(z)$). We can therefore construct $Z'$ greedily, by arbitrarily choosing $z \in Z$ to belong to $Z'$, and then removing all vertices $z' \in Z$ within face-distance $4$ of $z$, so that both $p$ and $m$ are injective on the domain $Z'$
	and the paths~$\{\beta_z\}_{z \in Z'}$ are pairwise disjoint. Since the number of vertices within face-distance $4$ of any vertex in $\G$ is at most $(\mbox{size of largest face in $\G$})^4$, we have proven inequality~\eqref{eq:bound-on-density}.
	
	We trace out the cycle $\gamma$ and argue, using planarity, that each time we encounter a vertex in $Z'$, there must be a new vertex in $\gamma \setminus Z$.  To that end, for each $j = 0, \dots, |Z'|$, we iteratively build sets $S_j \subseteq \gamma \setminus Z$ as follows. Start with $S_0 := \emptyset$. Given $j \geq 0$, consider three cases.
	\begin{enumerate}
		\item If $\prev_P(v_{i_j}) \not\in S_j$ then set $S_{j+1} := S_j \cup \{\prev_P(v_{i_j})\}$.
		\item Otherwise, let $k < j$ be the smallest index such that $\prev_P(v_{i_k}) = \prev_P(v_{i_j})$. If $k = j-1$, then set $S_{j+1} := S_j$.
		\item Otherwise, we set $S_{j+1} := S_j \cup \{\prev_M(v_{i_{j-1}})\}$.
	\end{enumerate}
	
	We claim that in each instance of case (1) or (3) we are only adding vertices which are not already in the set. The result for case (1) is immediate, so we need only argue that in case (3), we have $\prev_M(v_{i_{j-1}}) \not \in S_j$.
	
	Suppose for contradiction that $\prev_M(v_{i_{j-1}}) \in S_j$, so it was added in some previous step, say step $j'$. We therefore have $\prev_M(v_{i_{j'-1}}) = \prev_M(v_{i_{j-1}})$. The fact that step $j'$ fell through to case (3) means that there is $k' < j'-1$ with $\prev_P(v_{i_{k'}}) = \prev_P(v_{i_{j'}})$. From these equalities, we deduce the existence of a path $\xi \subseteq \G[M]$ from $m(v_{i_{j'-1}})$ to $m(v_{i_{j-1}})$ and a path $\zeta \subseteq \G[P]$ from $p(v_{i_{j'}})$ to $p(v_{i_{k'}})$. From these, we construct the paths\footnote{where $\alpha \cdot \beta$ denotes concatenation of paths (first following $\alpha$ and then following $\beta$) and the inverse denotes the path in reverse order} $\gamma_M := \beta_{v_{i_{j'-1}}} \cdot \xi \cdot \beta_{v_{i_{j-1}}}^{-1}$ from $v_{i_{j-1}}$ to $v_{i_{j'-1}}$ and $\gamma_P := \beta_{v_{i_{j'}}} \cdot \zeta \cdot \beta^{-1}_{v_{i_{k'}}}$ from $v_{i_{j'}}$ to $v_{i_{k'}}$. Since the paths $\{\beta_z\}_{z \in Z'}$ are pairwise disjoint and the paths $\xi \subseteq M$ and $\zeta \subseteq P$ are disjoint, we conclude that $\gamma_M$ and $\gamma_P$ are disjoint. But this is a contradiction to the Jordan curve theorem, because the endpoints of $\gamma_M$ and $\gamma_P$ alternate around $\gamma$ (as~$k' < j'-1 < j' < j-1$; see Figure~\ref{fig:third-case}).
	
	To conclude, we observe that case (2) cannot have more occurrences than case (1), so the total number of vertices we add is $|S_{|Z'|}| \geq \frac{1}{2}|Z'|$. Combining this bound with~\eqref{eq:bound-on-density}, we see that $\alpha_F := {\left(1+\frac{1}{2}{(\mbox{size of largest face in $\G$})}^{-4}\right)}^{-1}$ is satisfactory.
\end{proof}

\begin{figure}
	\begin{center}
	\includegraphics[width=0.5\textwidth]{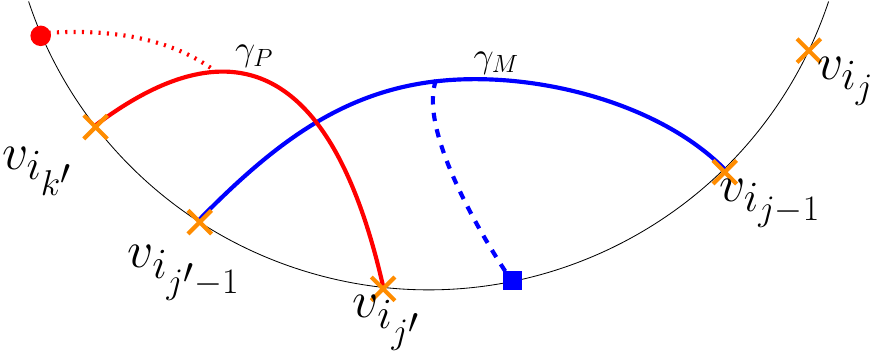}
	\caption{The contradiction in case 3 of the proof of Lemma~\ref{lemma:topological-lemma}. 
		The red circle is~$\prev_P(v_{i_{k'}}) = \prev_P(v_{i_{j'}})$ and the blue square is~$\prev_M(v_{i_{j-1}}) = \prev_M(v_{i_{j'-1}})$.
		The same color scheme as Figure~\ref{fig:inter-lines} is used
		and the paths~$\gamma_P$ and~$\gamma_M$ are red and blue solid lines respectively.
	}\label{fig:third-case}
	\end{center}
\end{figure}

\section{Many zeros implies identically zero}\label{sec:zero-case}
Let $\G^*$ be the dual multigraph of $\G$ whose vertices are the faces of $\G$ and whose edges correspond to adjacent faces of $\G$.
Denote by $\partial B_m$ the set of faces in $\G^*$ which are adjacent to both $B_m$ and its complement. Note that by periodicity of~$\mathcal{G}$, the cardinality of~$B_m$ grows as~$m^2$ and that of its boundary,~$\partial B_m$, as~$m$

We prove Theorem~\ref{theorem:zero-case} in steps. In the first two steps, we reduce to the case where $\G^*$ is a graph (not a multigraph) and there are few nonzeros on the boundary of $B_n$. The third step is an argument by contradiction: if there were a nonzero vertex in $B_{n/2}$, then by Lemma~\ref{lemma:topological-lemma} this would force many nonzeros on the boundary of $B_{n}$.
\smallskip

{\it Step 1: Reduction to case when $\G^*$ is a graph.} \\
This step is a technical reduction; for sufficiently connected graphs, the dual multigraph $\G^*$ is already a graph and we may proceed directly to the next step. We refer the interested reader to Diestel~\cite[Section 4.6]{diestel2017graph} for graph theory definitions and an exposition of planar graphs and duality.

First, we may assume without loss of generality that $\G$ is 2-edge connected\footnote{a graph is $k$-edge connected if it is impossible to make the graph disconnected by removing fewer than $k$ edges; a~$k$-edge connected component is a maximal~$k$-edge connected subgraph}. Indeed, since~$\G$ is connected and periodic, it contains one infinite~$2$-edge connected component
possibly connected via bridges to finite~$2$-edge connected components.   
By the maximum principle, any harmonic function is constant on finite 2-edge connected components of $\G$. Thus, we can delete these components and apply Theorem~\ref{theorem:zero-case} to the resulting graph.

Our next concern is that $\G^*$ can, in general, have multiple edges between the same two vertices. Supposing this was the case, there would be (\eg, by~\cite[Proposition 4.6.2]{diestel2017graph}) a finite induced subgraph $K$ of $\G$ which is 3-edge connected
and for which there are only two edges, $e_1$ and $e_2$, which connect $K$ to $\G \setminus K$. Let $\G'$ be the graph $\G$, with $K$ replaced by a single edge connecting the endpoints of $e_1$ and $e_2$ in $\G \setminus K$. The weight of this new edge is given by the effective conductance\footnote{see~\cite[Section 3.4]{Doyle_Snell_1984}} of $K$ between the endpoints. Then any harmonic function on $\G$ is also harmonic when restricted to $\G'$. Furthermore, since $\G$ is periodic, we can repeat this process for each finite 3-edge connected component of $\G$ and only remove a constant fraction of vertices from $\G$. Modifying the $\eps_0$ in Theorem~\ref{theorem:zero-case} appropriately according to this fraction, we see that proving the theorem for $\G'$ implies the result for $\G$, and therefore we assume that $\G^*$ is a graph without loss of generality.
\smallskip

{\it Step 2: Reduction to few non-zeros on the boundary.} \\
By the pigeonhole principle and the assumption that $f = 0$ on a $(1-\eps)$ portion of $B_{2 n}$, we may assume that there is some $\partial B_{m}$ for $m \in [1.5 n, 2 n]$ such the number of vertices for which $f(v) \neq 0$ and which are adjacent to a face in $\partial B_{m}$ is at most $C \eps n$.
\smallskip

{\it Step 3: Bounding the number of zeros.} \\
\begin{figure}
	\begin{center}
		\includegraphics[width=0.45\textwidth]{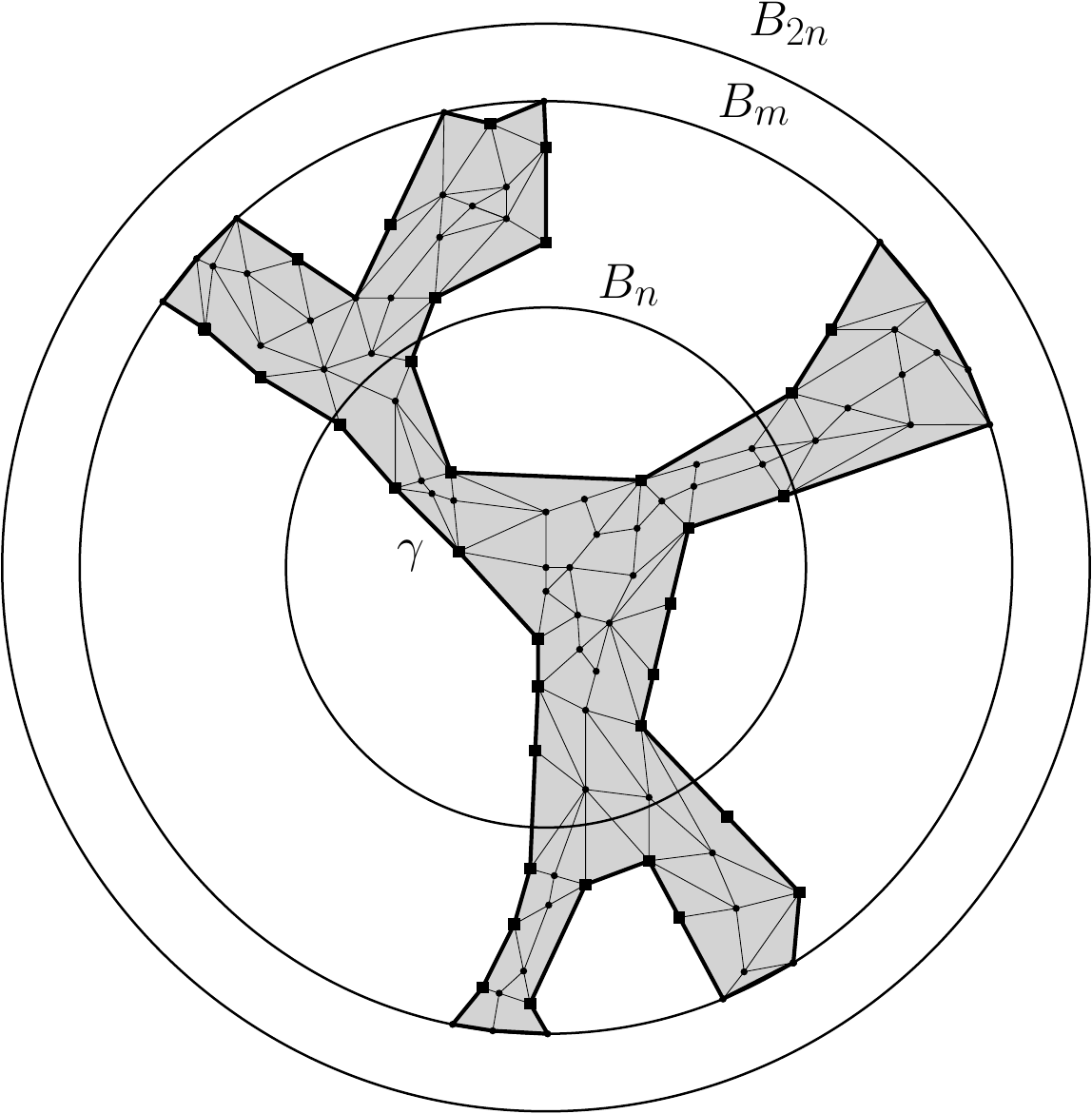}
	\end{center}
	\caption{A cartoon of Step 3; the cycle $\gamma$ is drawn as thick black lines, the vertices in $Z$ are drawn as boxes, and $D$ is the set of vertices inside $\gamma$.
	}\label{fig:step3figure}
\end{figure}
Let $m \in [1.5 n, 2 n]$ be given by Step 2.
Assume for contradiction that there is some $x_0 \in B_n$ with $f(x_0) \neq 0$. Let $N$ be the set of faces of $\G[B_{m}]$ which are adjacent to at least one vertex where $f$ is nonzero. Then the maximal connected component $K \subseteq N$ in $\G^*$ which contains every face adjacent to $x_0$ also contains, by the maximum principle, a face in $\partial B_{m}$.
Let $\gamma := \{v_0, v_1, \dots, v_k = v_0\}$ be a cycle in $\G$ around the boundary of $K \cup \{\text{vertices in finite connected components of $\G^* \setminus K$}\}$.

Let $D$ be the unique finite connected component of $\G \setminus \gamma$. First, we note that $D$ contains $x_0 \in B_n$ and a vertex adjacent to $\partial B_{m}$, so it has diameter at least $n/2$ and therefore its boundary, $\gamma$, has length at least $n/2$. Our goal is to find suitable sets $Z, P, M$ to apply Lemma~\ref{lemma:topological-lemma}. Let $Z := \gamma \setminus \{\mbox{vertices adjacent to $\partial B_m$}\}$ and let $P := \{x \in \gamma \cup D \mid f(x) > 0\}$ and $M := \{x \in \gamma \cup D \mid f(x) < 0\}$.

By maximality, $f(z) = 0$ for every vertex $z \in Z$. It is clear that $Z, P, M$ are disjoint. Furthermore, the first two hypotheses of Lemma~\ref{lemma:topological-lemma} are clearly satisfied. The fourth hypothesis of Lemma~\ref{lemma:topological-lemma} is satisfied by the maximum principle. It remains to check the third hypothesis: fix any $z \in Z$. Since $z$ is adjacent to a face $F \in K$, there is some vertex $w$ adjacent to $F$ on which $f(w) \neq 0$. Choose $w$ to be the vertex closest to $z$ with this property, and let $\beta = \{z = w_0, w_1, \dots, w_\ell\}$ be a shortest path such that $w_\ell$ is adjacent to $w$. By minimality of $\beta$, we have $f(w_\ell) = 0$ and therefore $w_\ell$ is also adjacent to some vertex $y$ with $\sgn f(y) = -\sgn f(w)$. Then we set $p(z)$ to be the vertex in $\{w, y\}$ on which $f$ is positive, and $m(z)$ to be the other vertex in $\{w, y\}$. Since $f$ vanishes on $\beta$, we see that the third hypothesis of Lemma~\ref{lemma:topological-lemma} is satisfied.

Applying Lemma~\ref{lemma:topological-lemma}, we conclude that $|Z| \leq \alpha_{\mathcal{F}} |\gamma|$.
Hence, at least $(1- \alpha_{\mathcal{F}} )$ vertices in $\gamma$ are adjacent to a face in $\partial B_{m}$.
Each of these faces must be adjacent to a vertex for which $f$ is nonzero.
However, since $|\gamma| \geq n/2$, this contradicts Step 2 for $\eps_0$ sufficiently small, completing the proof.\qed

\begin{remark}
	The argument shows that the density hypothesis in Theorem~\ref{theorem:zero-case} may be replaced,
	as in Step 2, by a bound on the number of non-zero vertices adjacent to $\partial B_{2 n}$.
	In particular, this shows that if there are infinitely many contours surrounding the origin for
	which the harmonic function has a high density of zeros, then the function must be zero identically.
\end{remark}

\section{Exponential upper bound}\label{sec:bounded-case}
For the proof of Theorem~\ref{theorem:uniformly-bounded}, we have the additional hypothesis of uniform ellipticity,
$0 < \lambda < \cond < \Lambda$. This allows us to exploit the following observation: if $f(x)$ is small
and there is a neighbor $y \sim x$ for which $f(y)$ has large magnitude,  then there is a different neighbor $z \sim x$ with large magnitude
of the opposite sign.  Equipped with this observation we argue along similar lines to the proof of Theorem~\ref{theorem:zero-case}.
\smallskip

{\it Step 1: An exponential bound.} \\
In this step we prove that there exists $\eps_1 \in (0,1)$ such that if $\eps \leq \eps_1$
\begin{equation} \label{eq:exponential-bound}
	\max_{B_n} |f| \leq C \alpha^n \, ,
\end{equation}
for constant exponent $\alpha$ depending only on $\mathcal{F}$ and the ellipticity ratio $\Theta := \Lambda/\lambda$.

By the pigeonhole principle, for each $\delta \in (0,1)$, there is a sufficiently large $K(\delta)$
and an integer $A \in [1, K^n]$ for which $|\{v \in \G[B_{2 n }] : A < |f(v)| < 2 \Theta \mbox{(max degree in $\G$)} A\}| < \delta n $.
We may then repeat the proof of Theorem~\ref{theorem:zero-case}, treating $\{v \in \G[B_{2 n }] \mid |f(v)| \leq  A\}$
as the zero set. The only difference is that there is a $\delta$-fraction of ``zeros''
which do not satisfy condition (3) of Lemma~\ref{lemma:topological-lemma}, but since $\delta$ can be
made arbitrarily small, we may discard those points from $Z$ before applying Lemma~\ref{lemma:topological-lemma} and we have \eqref{eq:exponential-bound}.
\smallskip

{\it Step 2: Improve the exponent by covering.} \\
Following the argument at the end of the proof of~\cite[Theorem ($\mathrm A'$)]{buhovsky2022discrete}, we upgrade the exponential bound from Step 1 to a bound involving $\eps$. Let $\delta = C^{-1} \eps_1^{-1/2} \eps^{1/2}$ and cover $B_{n}$ by $\delta^2$ balls of radius $\delta n$ contained in $B_{ 2n }$, decreasing $\eps_0$ if necessary.  Observe that for each such ball $B_{\delta n}(z_i)$, we have
$|v \in B_{\delta n}(z_i) \mid  |f(v)| > 1 |  < C \eps \delta^{-2} |B_{\delta n}|  = \eps_1 |B_{\delta n}|$
and hence, by applying Step 1 to each such ball, we have $\max_{B_n} |f| \leq \alpha^{(C^{-1} \eps_1^{-1/2} ) \eps^{1/2} n} $
which completes the proof. \qed

\section{A non-planar counterexample on the square lattice}\label{sec:counterexample}
We present a collection of periodic conductances with crossing edges on $\Z^2$ for which Theorem~\ref{theorem:zero-case} fails.
For each vertex $x$ such that $x_1 + x_2$ is even, assign (undirected) conductances as follows
\[
\cond( x, x - e_2) \, , \, \cond(x, x-e_1) := A_1 \qquad \cond(x, x + e_2 ) \,  , \, \cond(x, x + e_1) := A_2
\]
\[
\cond(x, x + (2 , 2 )) \, , \, \cond(x, x - (2 , 2 )) := A_3 \qquad \cond(x, x + (1, 1)) \, , \, \cond(x, x - (1, 1)) := A_4 \,,
\]
see Figure~\ref{fig:periodic-counterexample}. Denote this graph by $(\Z^2, \mathcal{E}(A_1,A_2,A_3,A_4))$.

\begin{figure}
	\begin{center}
		\includegraphics[width=0.2\textwidth]{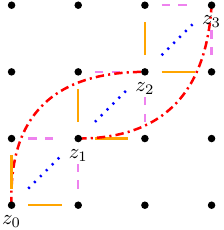}
	\end{center}
	\caption{Conductances $A_1,A_2,A_3,A_4$ are dashed violet lines, orange solid lines, red dash dotted lines, and blue dotted lines respectively.
	}\label{fig:periodic-counterexample}
\end{figure}

\begin{theorem}\label{theorem:counterexample}
	For each choice of positive $A_1 \neq A_2$
	and $A_3 > \frac{2 A_1^2 A_2^2}{{(A_1-A_2)}^2 (A_1+A_2)}$
	there is a choice of $A_4 > 0$ for which there is a harmonic
	function on $(\Z^2, \mathcal{E}(A_1,A_2,A_3,A_4))$ supported on the diagonal line $\{x_1 = x_2\}$.
\end{theorem}

\begin{proof}
	For the sequence $\{z_i\}_{i \in \Z}$ determined by the equalities $z_0 = 1$ and $A_1 z_{i}  + A_2 z_{i-1} = 0$,
	a short computation shows that
	the function $h(x_1, x_2) = 1_{\{ x_1 = x_2\}} z_{x_1}$
	is harmonic on $(\Z^2, \mathcal{E}(A_1,A_2,A_3,A_4))$ for some $A_4 > 0$.
\end{proof}

\appendix

\section{Exponential lower bound}

We follow the arguments of~\cite{buhovsky2022discrete} and prove the following exponential lower bound, which, when combined with Theorem~\ref{theorem:uniformly-bounded}, implies Theorem~\ref{theorem:liouville}. Let $Q_N$ denote the square centered at 0 of side length $2 N+1$, that is,~$Q_N := \{ (x,y) \in \mathbb{R}^2 : \max(|x|, |y| ) \leq N\}$.
\begin{theorem}[Theorem~(B) in~\cite{buhovsky2022discrete}]
	\label{theorem:lower-bound}
	There is some $b > 0$ such that the following holds. If $\varepsilon > 0$ is sufficiently small, $N$ is sufficiently large, $\max_{Q_{\lfloor \sqrt{N} \rfloor}} f \geq 2$, and
	\begin{equation*}
		\frac{|\{|f| \leq 1\} \cap Q_K|}{|Q_K \cap \V|} \geq 1-\eps,
	\end{equation*}
	for each $K \in [\sqrt{N}, 2N]$, then
	\begin{equation}
		\label{eq:exp-growth-statement}
		\max_{Q_{N}} |f| \geq \exp(bN) \, .
	\end{equation}
\end{theorem}
The proof in~\cite{buhovsky2022discrete} relied on a discrete three-ball inequality and our only modification is to prove this inequality in the more general setting of a periodic graph. The argument is effectively a simplified version of that in~\cite{armstrong2023large}.

We first approximate a harmonic function on $\G$ by a polynomial with periodic coefficients.
\begin{lemma}
	\label{lemma:poly-approx}
	For any $\alpha > 0$, there is a constant $c(\alpha) > 0$ such that, for every sufficiently large $R > 0$, natural number $m \leq c R$, and $v \in \V$, if $f$ is discrete harmonic in $Q_{3 R}$, then there exists a polynomial $p$ of degree $m$ such that
	\begin{equation}
		\label{eq:poly-approx}
		\|f-p\|_{L^\infty(Q_{cR} \cap (v + \mathcal{L}))} \leq \alpha^m \|f\|_{L^\infty(Q_{3 R})} \, .
	\end{equation}
\end{lemma}
\begin{proof}
	Fix generators $e_1, e_2 \in \mathcal{L}$. For $i=1,2$, let $D_i$ denote the forward-difference operator
	$D_i f (x) := f(x+e_i)-f(x) $.
	We write $D := (D_1, D_2)$.
	
	Iterating the discrete Caccioppoli inequality (see, \eg, \cite[Proposition 12]{MR3395463}) 
	and applying a discrete Moser estimate, \cite[Proposition 5.3]{MR1425544}, we obtain
	\begin{equation}
		\label{eq:derivative-bound}
		R \|D^m f \|_{L^\infty(Q_R)} \leq C \|D^m f \|_{L^2(Q_{2 R})} \leq {\left(\frac{Cm}{R}\right)}^m\| f \|_{L^2(Q_{3R})} \quad \forall m \leq c R \, .
	\end{equation}
	Here we use the fact that if $f $ is harmonic then $D f$ is also harmonic. This is the only place in our argument where $\mathcal{L}$-invariance of $\cond$ is used.
	
	Given $v \in \V$, choose $v_0 \in v + \mathcal{L}$ to be closest to the origin, breaking ties arbitrarily. Let $p \in \R[x_1, x_2]$ be the unique polynomial of degree $m$ such that
	$D^k p(v_0) = D^k f(v_0)$ for all $k \in [0, m]$.
	For any $z \in Q_{cR} \cap (v + \mathcal{L})$, integrating $m+1$ times from $v_0$ to $z$ yields
	\begin{equation*}
		|f(z)-p(z)| \leq \frac{{\left(cR\right)}^{m+1}}{(m+1)!} \sup_{Q_{cR + C(m + 1)} \cap (v + \mathcal{L})}|D^{m+1} f | \, .
	\end{equation*}
	The desired conclusion follows by applying~\eqref{eq:derivative-bound} at order $m+1$ and choosing $c(\alpha) > 0$ sufficiently small.
\end{proof}

\begin{prop}
	\label{prop:three-ball}
	There is some $\eps = \eps(\G)$ such that, if
	\begin{equation*}
		\frac{|\{|f| \leq 1\} \cap Q_N|}{|Q_N \cap \V|} \geq 1-\eps,
	\end{equation*}
	and $|f| \leq M$ on $Q_{4N}$, then
	\begin{equation}
		\label{eq:three-ball-statement}
		\max_{Q_{2N}} |f| \leq CM^{1/2} + C\exp(-cN)M \, ,
	\end{equation}
	where $c, C > 0$ are constants.
\end{prop}
\begin{proof}
	We closely follow the proof of~\cite[Theorem~3.1]{buhovsky2022discrete}, making a small modification to use Lemma~\ref{lemma:poly-approx} instead of the exact formula for the Poisson kernel on squares in $\Z^2$.
	
	As in~\cite{buhovsky2022discrete}, it suffices to prove the following statement which implies the desired result by a routine covering argument.
	
	\emph{There is some $\varepsilon > 0$ and $k \in \N$ such that, if $f$ is discrete harmonic with $|f| \leq M$ on $Q_{kN}$ and $|f| \leq 1$ on at least $(1-\eps)|Q_N \cap \V|$  vertices in $Q_{N}$, then $|f| \leq C(M^{1/2} + \exp(-cN)M)$ on $Q_{2N}$.}
	
	By choosing $\eps$ sufficiently small, it suffices to prove this statement only for vertices in the translated lattice $v + \mathcal{L}$ for some fixed $v \in \V$. Fix such a $v$ and choose vectors $e_1, e_2$ that generate $\mathcal{L}$. Without loss of generality, assume that $e_1 = (1, 0)$ and $e_2 = (0, 1)$. Consider an integer $t \in [-N, N]$ where $|f(v + se_1 + te_2)| \leq 1$ for at least half of integers $s \in [-N, N]$. We estimate $\sup_{s \in [-2N, 2N]} |f(v + se_1 + te_2)|$ and then repeat, propagating the bounds from the horizontal direction to the vertical direction.
	
	Let $\beta < 0$ and $\alpha \in (0, 1)$ be constants chosen at the end of the proof. At each occurrence of these parameters, we will make a note of any necessary conditions.
	
	We consider two cases, depending on the size of $M$.
	\begin{enumerate}
		\item
		Suppose $M\exp(\beta N) \leq 1$ and choose $\gamma > 0$ to be minimal such that $\gamma N$ is an integer and
		\begin{equation*}
			\alpha^{\gamma N}M \leq 1 \, .
		\end{equation*}
		By the assumption of this case and minimality of $\gamma$, we conclude that $\gamma \leq \frac{\beta}{\log \alpha} + \frac{1}{N}$.
		Let $p$ be the polynomial given by Lemma~\ref{lemma:poly-approx} of degree $m=\gamma N$, which satisfies, for $k$ sufficiently large,
		\begin{equation*}
			\|f-p\|_{L^\infty(Q_{N} \cap (v + \mathcal{L}))} \leq \alpha^{\gamma N}M \, .
		\end{equation*}
		By a discrete version of Remez's inequality~\cite[Corollary~2.2]{buhovsky2022discrete}, we deduce
		\begin{equation*}
			|p(v + se_1 + te_2)| \leq 2{\left(\frac{16N}{(1-\gamma)N}\right)}^{\gamma N}, \quad \forall s \in [-2N, 2N] \, .
		\end{equation*}
		Combining the previous three displays, for each $s \in [-2N, 2N]$ we get
		\begin{equation*}
			|f(v + se_1 + te_2)| \leq 2{\left(\frac{16N}{(1-\gamma)N}\right)}^{\gamma N} + 1 \leq 3 \alpha^{-\gamma N/2} \leq \frac{3}{\alpha} M^{1/2} \, ,
		\end{equation*}
		where the second inequality above holds as long as $\log \frac{16}{1 - \gamma} \leq -\frac{1}{2}\log \alpha$ and the third follows by minimality of $\gamma$. In turn, $\log \frac{16}{1 - \gamma} \leq -\frac{1}{2}\log \alpha$ holds if $\gamma \leq \frac{3}{4}$ (which holds for $N \geq 4$ when $\frac{\beta}{\log \alpha} \leq \frac12$) and $\log 64 \leq -\frac{1}{2} \log \alpha$.
		\item
		
		Otherwise, assume that $\delta := M\exp(\beta N) > 1$. In this case, approximate $u$ instead by a polynomial of degree $\frac{1}{2}N$, given by Lemma~\ref{lemma:poly-approx}, with error
		\begin{equation*}
			\|f-p\|_{L^\infty(Q_{N} \cap (v + \mathcal{L}))} \leq \exp(\beta N)M \, .
		\end{equation*}
		By the assumption of this case and the above inequality, we have $|p(v + se_1 + te_2)| \leq 2\delta$ for at least half of the $s \in [-N, N]$. Applying~\cite[Corollary~2.2]{buhovsky2022discrete} again, we get, for each $s \in [-2N, 2N]$,
		\begin{equation*}
			|p(v + se_1 + te_2)| \leq 2\delta {\left(\frac{32N}{N}\right)}^{N/2} \, .
		\end{equation*}
		Combining the previous two displays, we get, for each $s \in [-2N, 2N]$,
		\[
		|f(v + se_1 + te_2)| \leq {32}^{N/2}2\delta + \exp(\beta N)M \leq C\exp(-cN)M \, ,
		\]
		where we use the fact that $\beta < -\log(32)/2$.
	\end{enumerate}
	We observe that the desired constraints are satisfied for $\beta := -\log 32$ and $\alpha := 2^{-12}$.
\end{proof}

\begin{proof}[Proof of Theorem~\ref{theorem:lower-bound}]
	Follow the proof of Theorem~(B) in~\cite{buhovsky2022discrete}, replacing each reference to Theorem~3.1 (the discrete three-circle theorem) with Proposition~\ref{prop:three-ball}, which is the same statement generalized to periodic planar graphs.
\end{proof}

\section*{Acknowledgments}
Thank you to Scott Armstrong, Paul Dario, Lionel Levine, and Charles Smart for useful discussions and suggestions.
A.B. was partially supported by NSF grant DMS-2202940 and a Stevanovich fellowship.
W.C. was partially supported by NSF grant DMS-2303355.
S.G. was partially supported by NSF grants DMS-1855688,
DMS-1945172, and a Sloan Fellowship.

{
	\bibliographystyle{alpha}
	\bibliography{refs}
}

\end{document}